\documentclass[11pt,leqno]{article}
\usepackage{url}
\usepackage{enumitem}
\usepackage[doc]{optional}
\usepackage{color}
\usepackage{float}
\usepackage{soul}
\usepackage{graphicx}
\usepackage{lmodern,bm}

\definecolor{labelkey}{rgb}{0,0.08,0.45}
\definecolor{refkey}{rgb}{0,0.6,0.0}
\definecolor{Brown}{rgb}{0.45,0.0,0.05}
\definecolor{lime}{rgb}{0.00,0.8,0.0}
\definecolor{lblue}{rgb}{0.5,0.5,0.99}
\definecolor{lblue}{rgb}{0.8,0.85,1.00}
\definecolor{anotherblue}{rgb}{.8, .8,1}
\definecolor{violet}{rgb}{0.9,0.6,0.9}
\definecolor{greenyellow}{rgb}{0.53,0.99,0.18}
\definecolor{Lyellow}{rgb}{0.99,0.99,0.87}
\definecolor{Lgray}{rgb}{0.93,0.93,0.93}

\usepackage{mathpazo}
\usepackage{amsmath}
\usepackage{amssymb}
\usepackage{theorem}

\usepackage{rotating}
\usepackage{graphicx}
\usepackage{graphics}
\usepackage{longtable}
\usepackage{fancybox}
\usepackage{mathrsfs}

\makeatletter
\newcommand{\vast}{\bBigg@{4}}
\newcommand{\Vast}{\bBigg@{5}}
\makeatother

\newcommand{\kkk}{\ensuremath{{k\in{\mathbb N}}}}

\newcommand{\menge}[2]{\big\{{#1}~\big |~{#2}\big\}}
\newcommand{\mmenge}[2]{\bigg\{{#1}~\bigg |~{#2}\bigg\}}
\newcommand{\mmmenge}[2]{\Bigg\{{#1}~\Bigg |~{#2}\Bigg\}}

\newcommand{\MMenge}[2]{\vast\{{#1}~\vast|~{#2}\vast\}}

\newcommand{\To}{\ensuremath{\rightrightarrows}}

\newcommand{\fenv}[1]%
{\ensuremath{\,\overrightarrow{\operatorname{env}}_{#1}}}
\newcommand{\benv}[1]%
{\ensuremath{\,\overleftarrow{\operatorname{env}}_{#1}}}
\newcommand{\emp}{\ensuremath{\varnothing}}

\newcommand{\scal}[2]{\left\langle{#1},{#2}  \right\rangle}
\newcommand{\tscal}[2]{\langle{#1},{#2}  \rangle}

\newcommand{\exi}{\ensuremath{\exists\,}}
\newcommand{\zeroun}{\ensuremath{\left]0,1\right[}}
\newcommand{\RR}{\ensuremath{\mathbb R}}
\newcommand{\ZZ}{\ensuremath{\mathbb Z}}
\newcommand{\RP}{\ensuremath{\mathbb R}_+}

\newcommand{\RPP}{\ensuremath{\mathbb R}_{++}}

\newcommand{\NN}{\ensuremath{\mathbb N}}

\newcommand{\gr}{\ensuremath{\operatorname{gr}}}

\newcommand{\inte}{\ensuremath{\operatorname{int}}}

\newcommand{\card}{\ensuremath{\operatorname{card}}}
\newcommand{\bd}{\ensuremath{\operatorname{bdry}}}

\newcommand{\ran}{\ensuremath{\operatorname{ran}}}
\newcommand{\rank}{\ensuremath{\operatorname{rank}}}

\newcommand{\cone}{\ensuremath{\operatorname{cone}}}
\newcommand{\lspan}{\ensuremath{\operatorname{span}}}
\newcommand{\aff}{\ensuremath{\operatorname{aff}}}
\newcommand{\pa}{\ensuremath{\operatorname{par}}}

\newcommand{\minf}{\ensuremath{-\infty}}
\newcommand{\pinf}{\ensuremath{+\infty}}

\newcommand{\ball}[2]{\operatorname{ball}({#1};{#2})}
\newcommand{\sphere}[2]{\operatorname{sphere}({#1};{#2})}

\newcommand{\wt}[1]{\widetilde{#1}}
\newcommand{\nc}[2]{N^{#2}_{#1}}
\newcommand{\fnc}[1]{N^{\text{\rm Fr\'e}}_{#1}}
\newcommand{\cnc}[1]{N^{\text{\rm conv}}_{#1}}
\newcommand{\pnX}[1]{N^{\text{\rm prox}}_{#1}} 
\newcommand{\pn}[2]{\widehat{N}^{#2}_{#1}} 

\makeatletter 
\renewcommand\p@enumii{}
\makeatother

\def\disp{\displaystyle}
\def\ve{\varepsilon}
\def\dd{\delta}
\def\dm{{{n}}}

\def\dn{\downarrow}

\def\disp{\displaystyle}

\newcommand{\mcA}{\ensuremath{\mathcal A}}
\newcommand{\mcB}{\ensuremath{\mathcal B}}
\newcommand{\mcJ}{\ensuremath{\mathcal J}}
\newcommand{\mcC}{\ensuremath{\mathcal C}}

\newcommand{\supp}{\ensuremath{\operatorname{supp}}}
\newcommand{\sgn}{\ensuremath{\operatorname{sgn}}}

{\begin{list}{}{%
\settowidth{\labelwidth}{\textrm{#1~}}%
\setlength{\leftmargin}{\labelwidth+\labelsep}}}
{\end{list}}
\newtheorem{theorem}{Theorem}[section]
\newtheorem{lemma}[theorem]{Lemma}
\newtheorem{corollary}[theorem]{Corollary}
\newtheorem{proposition}[theorem]{Proposition}
\newtheorem{definition}[theorem]{Definition}
\theoremstyle{plain}{\theorembodyfont{\rmfamily}
}
\theoremstyle{plain}{\theorembodyfont{\rmfamily}
}
\theoremstyle{plain}{\theorembodyfont{\rmfamily}
}
\theoremstyle{plain}{\theorembodyfont{\rmfamily}
\newtheorem{example}[theorem]{Example}}

\theoremstyle{plain}{\theorembodyfont{\rmfamily}
\newtheorem{remark}[theorem]{Remark}}

\newcommand{\boxedeqn}[1]{%
    \[\fbox{%
        \addtolength{\linewidth}{-2\fboxsep}%
        \addtolength{\linewidth}{-2\fboxrule}%
        \begin{minipage}{\linewidth}%
        \begin{equation}#1\\[+4mm]\end{equation}%
        \end{minipage}%
      }\]%
  }

\newcommand{\dotcup}{\ensuremath{\mathaccent\cdot\cup}}

\newcounter{count}

\allowdisplaybreaks 

\begin{document}

\title{\textrm{
Restricted normal cones and\\ sparsity optimization with affine constraints  }}

\author{
Heinz H.\ Bauschke\thanks{Mathematics, University of British
Columbia, Kelowna, B.C.\ V1V~1V7, Canada. E-mail:
\texttt{heinz.bauschke@ubc.ca}.},
D.\ Russell Luke\thanks{Institut f\"ur Numerische und Angewandte Mathematik,\
Universit\"at G\"ottingen,\ Lotzestr.~16--18, 37083 G\"ottingen, Germany. E-mail: \texttt{r.luke@math.uni-goettingen.de}.}, Hung M.\
Phan\thanks{Mathematics, University of British Columbia, Kelowna,
B.C.\ V1V~1V7, Canada. E-mail:  \texttt{hung.phan@ubc.ca}.}, ~and
Xianfu\ Wang\thanks{Mathematics, University of British Columbia,
Kelowna, B.C.\ V1V~1V7, Canada. E-mail:
\texttt{shawn.wang@ubc.ca}.}}

\date{May 2, 2012}

\maketitle

\vskip 8mm

\begin{abstract} \noindent
The problem of finding a vector with the fewest nonzero elements 
that satisfies an underdetermined system of linear equations is 
an NP-complete problem that is typically solved numerically via 
convex heuristics or nicely-behaved nonconvex relaxations.  In   
this paper we consider the elementary method of alternating 
projections (MAP) for solving the sparsity optimization problem 
without employing convex heuristics.  In a parallel paper we 
recently introduced the restricted normal cone which generalizes 
the classical Mordukhovich normal cone and reconciles some fundamental 
gaps in the theory of sufficient conditions for local linear convergence 
of the MAP algorithm.  We use the restricted normal cone together with the 
notion of superregularity, which is naturally satisfied for the 
affine sparse optimization problem, to obtain local linear convergence 
results with estimates for the radius of convergence of the MAP algorithm
applied to sparsity optimization with an affine constraint.  
\end{abstract}




\section{Introduction}
We consider the problem 
of sparsity optimization with affine constraints:
\begin{equation}
\label{e:0314a}
{\mathrm{minimize}}\;\;\|x\|_0~ \mbox{ subject to } Mx=p
\end{equation}
where $m$ and $n$ are integers such that  $1\leq m<\dm$,
$M$ is a real $m$-by-$\dm$ matrix, denoted $M\in\RR^{m\times\dm}$, and 
$\|x\|_0 := \sum_{j=1}^\dm |\sgn(x_j)|$
counts\footnote{We set $\sgn(0):=0$.} the number of nonzero entries of 
real vectors $x$ of length $\dm$, denoted by $x\in \RR^\dm$.

If there is some a priori bound on the desired sparsity of the
solution, represented by an integer $s$,
where $1\leq s\leq \dm$, then one can relax \eqref{e:0314a} to the feasibility problem
\begin{equation}\label{e:sparse feas}
   \mbox{find } c\in A\cap B,
\end{equation}
where
\begin{equation}\label{e:AB}
A:=\menge{x\in \RR^\dm}{\|x\|_0\leq s}
\quad\text{and}\quad B:=\menge{x\in \RR^\dm}{Mx=p}.
\end{equation}
The sparsity subspace associated with $a=(a_1,\ldots,a_\dm)\in \RR^\dm$ is
\begin{equation}\label{e:supp}
  \supp(a):=\menge{x\in \RR^\dm}{x_j=0 \text{~whenever~} a_j=0}.
\end{equation}
Also, we define
\begin{equation}
\label{e:120411a}
I\colon \RR^\dm\to \{1,\ldots,\dm\}\colon
x\mapsto \menge{i\in\{1,\ldots,\dm\}}{x_i\neq 0},
\end{equation}
and we denote the $i^\text{th}$ standard unit vector by $e_i$
for every $i\in\{1,\ldots,\dm\}$.

Problem \eqref{e:0314a} is in general NP-complete \cite{Natarajan95} and 
so convex and nonconvex relaxations are typically employed for its solution.  For a primal-dual 
convex strategy see \cite{BorLuke10b};  for relaxations to $\ell_p ~(0<p<1)$ see \cite{LaiWang10};  
see \cite{BrucksteinDonohoElad09} for a comprehensive review and applications.
In this paper we apply recent tools developed by the authors in \cite{BLPW12a}
to prove local linear  convergence of an elementary algorithm applied to the feasibility 
formulation of the problem \eqref{e:sparse feas}, that is, we do not use convex heuristics or 
conventional smooth relaxations.  The key to our results 
is a new normal cone called the \emph{restricted normal cone}.  A central feature of 
our approach is the decomposition of the original nonconvex set into collections
of simpler (indeed, linear) sets which can be treated separately.   Ours is not the first 
result on local linear convergence for sparsity optimization with affine constraints.  
Indeed the problem was considered more than twenty years ago by Combettes and Trussell 
who show local convergence of alternating projections \cite{CombTrus90}.
The problem 
was recently used to illustrate the application of analytical tools developed in \cite{LM} and \cite{Luke12}.  Other
approaches that also yield convergence results for different algorithms can be found in 
\cite{ABRS} and \cite{BeckTeboulle11}, 
with the latter of these being notable in that they obtain {\em global} convergence results with additional 
assumptions (restricted isometry)  that we do not consider here.  The novelty of the results we report here, 
based principally on the works \cite{LM}, \cite{LLM} and \cite{BLPW12a}, is that we obtain not only 
convergence rates but also radii of convergence when all conventional sufficient conditions
for local linear convergence, notably those of \cite{LM} and \cite{LLM},  fail.  In this sense, our 
criteria for convergence are more robust and yield richer information than other available notions.  

The remainder of the paper is organized as follows.
In Section~\ref{s:prelim}, we define the restricted normal cones and corresponding 
constraint qualifications for sets and collections of sets first introduced in
\cite{BLPW12a} as well as the notion of superregularity introduced in 
\cite{LLM} adapted to the restricted normal cones.  A few of the many
properties of these objects developed in \cite{BLPW12a} are restated in preparation for
Section~\ref{s:russell} where we apply these tools to a convergence analysis of the 
method of alternating projections (MAP) for the problem of 
finding a vector $c\in \mathbb{R}^n$ satisfying an affine 
constraint and having sparsity no greater than some a priori bound, that is, 
we solve \eqref{e:sparse feas} for $A$ and $B$ defined by \eqref{e:AB}.
Given a starting point $b_{-1}\in X$,
MAP sequences $(a_k)_\kkk$ and $(b_k)_\kkk$ are generated as
follows:
\begin{equation}\label{e:MAP}
(\forall\kkk)\qquad
a_{k} := P_Ab_{k-1},\quad b_k := P_Ba_k.
\end{equation}
We do not attempt to review the history of the MAP, its many extensions, and its rich
and convergence theory;
the interested reader is referred to, e.g.,
\cite{BC2011},
\cite{CensorZenios},
\cite{Deutsch},
and the references therein.
We consider the MAP iteration to be a prototype for more
sophisticated approaches, both of projection type or more generally 
subgradient algorithms, hence our focus on this simple algorithm.  

\subsection*{Notation}

Our notation is standard and follows largely
\cite{BC2011},\cite{Borzhu05},
\cite{Boris1},
\cite{Rock70},
and \cite{Rock98} to which the reader is referred for more background on variational
analysis.  Throughout this paper, we assume that $X=\RR^n$ 
 with inner product $\scal{\cdot}{\cdot}$, induced norm $\|\cdot\|$,
and induced metric $d$. 
The real numbers are $\RR$,
the integers are $\ZZ$, and
$\NN := \menge{z\in\ZZ}{z\geq 0}$.
Further, $\RP := \menge{x\in\RR}{x\geq 0}$,
$\RPP := \menge{x\in \RR}{x>0}$.
Let $R$ and $S$ be subsets of $X$.
Then the closure of $S$ is $\overline{S}$,
the interior of $S$ is $\inte(S)$,
the boundary of $S$ is $\bd(S)$,
and the smallest affine and linear subspaces containing $S$
are $\aff S$ and $\lspan S$, respectively.
If $Y$ is an affine subspace of $X$, then $\pa Y$ is the
unique linear subspace parallel to $Y$. 
The negative polar cone of $S$ is
$S^\ominus=\menge{u\in X}{\sup\scal{u}{S}\leq 0}$.
We also set $S^\oplus := -S^\ominus$
and $S^\perp := S^\oplus \cap S^\ominus$.
We also write
$R\oplus S$ for $R+S:=\menge{r+s}{(r,s)\in R\times S}$
provided that $R\perp S$, i.e., $(\forall (r,s)\in R\times S)$
$\scal{r}{s}=0$.
We write $F\colon X\To X$, if $F$ is a mapping from $X$ to its power set,
i.e., $\gr F$, the graph of $F$, lies in $X\times X$.
Abusing notation slightly, we will write $F(x) = y$ if $F(x)=\{y\}$.
A nonempty subset $K$ of $X$ is a cone
if $(\forall\lambda\in\RP)$ $\lambda K := \menge{\lambda k}{k\in
K}\subseteq K$.
The smallest cone containing $S$ is denoted $\cone(S)$;
thus, $\cone(S) := \RP\cdot S := \menge{\rho s}{\rho\in\RP,s\in S}$
if $S\neq\varnothing$ and $\cone(\varnothing):=\{0\}$.
If $z\in X$ and $\rho\in\RPP$, then
$\ball{z}{\rho} := \menge{x\in X}{d(z,x)\leq \rho}$ is
the closed ball centered at $z$ with radius $\rho$ while
$\sphere{z}{\rho} := \menge{x\in X}{d(z,x)= \rho}$ is
the (closed) sphere centered at $z$ with radius $\rho$.
If $u$ and $v$ are in $X$,
then $[u,v] := \menge{(1-\lambda)u+\lambda v}{\lambda\in [0,1]}$ is
the line segment connecting $u$ and $v$.

\section{Foundations}
\label{s:prelim}
We review in this section some of the fundamental tools used in the analysis of 
projection algorithms, and in particular MAP, for the solution of feasibility problems 
like \eqref{e:sparse feas}.  The tools below are intended for more general situations
where the sets $A$ and $B$ might admit decompositions into unions 
of sets, in which case we consider the feasibility problem
\begin{equation}\label{e:gen feas}
\mbox{ find }\quad c\in 
\bigg(\bigcup_{i\in I}A_i\bigg) \cap \bigg(\bigcup_{j\in J}B_j\bigg)
\end{equation}
Central to the convergence analysis of the MAP algorithm for solving 
\eqref{e:gen feas} is the notion of regularity of the intersection and 
the regularity of neighborhoods of the intersection.  These ideas are
developed in detail in \cite{BLPW12a}.  We review the main points 
relevant to our application here.

Normal cones are used to provide information about the orientation and local geometry of subsets of $X$.  
There are many species of normal cones, the key ones for our purposes are defined here.
In addition to the 
classical notions ({\em proximal, Fr\'echet, Mordukhovich}) we define the 
{\em restricted} normal cone introduced and developed in \cite{BLPW12a}.
\begin{definition}[normal cones]
\label{d:NCone}
Let $A$ and $B$ be nonempty subsets of $X$, and let $a$ and
$u$ be in $X$.
If $a\in A$, then various normal cones of $A$ at $a$ are defined as follows:
\begin{enumerate}
\item
\label{d:pnB} The \emph{$B$-restricted proximal normal cone} of $A$
at $a$ is
\begin{equation}\label{e:pnB}
\pn{A}{B}(a):= \cone\Big(\big(B\cap P_A^{-1}a\big)-a\Big) =
\cone\Big(\big(B-a\big)\cap \big(P_A^{-1}a-a\big)\Big).
\end{equation}
\item
\label{d:pnX} The (classical) \emph{proximal normal cone} of $A$ at
$a$ is
\begin{equation}\label{e:pnX}
\pnX{A}(a):=\pn{A}{X}(a)= \cone\big(P_A^{-1}a-a\big).
\end{equation}
\item\label{d:nc}
The \emph{$B$-restricted normal cone} $\nc{A}{B}(a)$ is implicitly
defined by $u\in\nc{A}{B}(a)$ if and only if there exist sequences
$(a_k)_\kkk$ in $A$ and $(u_k)_\kkk$ in $\pn{A}{B}(a_k)$ such that
$a_k\to a$ and $u_k\to u$.
\item\label{d:fnc} The \emph{Fr\'{e}chet normal cone}
$\fnc{A}(a)$ is implicitly defined by $u\in \fnc{A}(a)$ if and only
if $(\forall \ve>0)$ $(\exi\dd>0)$ $(\forall x\in A\cap
\ball{a}{\delta})$ $\scal{u}{x-a}\leq\ve\|x-a\|$.
\item\label{d:cnc}
The \emph{convex normal from convex analysis} $\cnc{A}(a)$ is
implicitly defined by $u\in\cnc{A}(a)$ if and only if
$\sup\scal{u}{A-a}\leq 0$.
\item\label{d:Mnc}
The \emph{Mordukhovich normal cone} $N_A(a)$ of $A$ at $a$ is
implicitly defined by $u\in N_A(a)$ if and only if there exist
sequences $(a_k)_\kkk$ in $A$ and $(u_k)_\kkk$ in $\pnX{A}(a_k)$
such that $a_k\to a$ and $u_k\to u$.
\end{enumerate}
If $a\notin A$, then all normal cones are defined to be empty.
\end{definition}

The following elementary calculus rules are a restatement of 
\cite[Proposition~3.7]{BLPW12a}.
\begin{proposition}
\label{p:elementary} Let $A$, $A_1$, $A_2$, $B$, $B_1$, and $B_2$ be
nonempty subsets of $X$, let $c\in X$,
and suppose that $a\in A\cap A_1\cap A_2$.
Then the following hold:
\begin{enumerate}
\item\label{p:ele-i}
If $A$ and $B$ are convex, then $\pn{A}{B}(a)$ is convex.
\item\label{p:ele-ii}
$\pn{A}{B_1\cup B_2}(a)=\pn{A}{B_1}(a)\cup\pn{A}{B_2}(a)$ and
$\nc{A}{B_1\cup B_2}(a)=\nc{A}{B_1}(a)\cup\nc{A}{B_2}(a)$.
\item\label{p:ele-iii}
If $B\subseteq A$, then $\pn{A}{B}(a)=\nc{A}{B}(a)=\{0\}$.
\item\label{p:ele-iv}
If $A_1\subseteq A_2$, then $\pn{A_2}{B}(a)\subseteq
\pn{A_1}{B}(a)$.
\item\label{p:ele-v}
$-\pn{A}{B}(a)=\pn{-A}{-B}(-a)$, $-\nc{A}{B}(a)=\nc{-A}{-B}(-a)$,
and $-N_{A}(a)=N_{-A}(-a)$.
\item\label{p:ele-vi}
$\pn{A}{B}(a)=\pn{A-c}{B-c}(a-c)$ and $\nc{A}{B}(a)=\nc{A-c}{B-c}(a-c)$.
\end{enumerate}
\end{proposition}

The {\em constraint qualification-}, or {\em CQ-number} defined next is built upon the 
normal cone and quantifies classical notions of constraint qualifications for set intersections
that indicate sufficient regularity of the intersection.    
\begin{definition}[(joint) CQ-number]
\label{d:CQn}
Let $A$, $\wt{A}$, $B$, $\wt{B}$, be nonempty subsets of $X$,
let $c\in X$, and let $\dd\in\RPP$.
The \emph{CQ-number} at $c$ associated with
$(A,\wt{A},B,\wt{B})$ and $\dd$ is
\begin{equation}
\label{e:CQn}
\theta_\dd:=\theta_\dd\big(A,\wt{A},B,\wt{B}\big)
:=\sup\mmenge{\scal{u}{v}}
{\begin{aligned}
&u\in\pn{A}{\wt{B}}(a),v\in-\pn{B}{\wt{A}}(b),\|u\|\leq 1, \|v\|\leq 1,\\
&\|a-c\|\leq\dd,\|b-c\|\leq\dd.
\end{aligned}}.
\end{equation}
The \emph{limiting CQ-number} at $c$ associated with
$(A,\wt{A},B,\wt{B})$ is
\begin{equation}\label{e:lCQn}
\overline\theta:=\overline\theta\big(A,\wt{A},B,\wt{B}\big)
:=\lim_{\dd\dn0}\theta_\dd\big(A,\wt{A},B,\wt{B}\big).
\end{equation}
For  nontrivial collections\footnote{The collection $(A_i)_{i\in
I}$ is said to be \emph{nontrivial} if $I\neq\varnothing$.} 
$\mathcal{A} := (A_i)_{i\in I}$, $\wt{\mcA}:=(\wt{A}_i)_{i\in
I}$, $\mathcal{B} := (B_j)_{j\in J}$, $\wt{\mcB}:=(\wt{B}_j)_{j\in
J}$ of
nonempty subsets of $X$, the \emph{joint-CQ-number} at
$c\in X$ associated with $(\mathcal{A},\wt{\mcA},\mathcal{B},\wt{\mcB})$
and $\dd>0$ is
\begin{equation}
\label{e:jCQn}
\theta_\dd=\theta_\dd\big(\mathcal{A},
\wt{\mcA},\mathcal{B},\wt{\mcB}\big):=\sup_{(i,j)\in I\times
J}\theta_\dd\big(A_i,\wt{A}_i,B_j,\wt{B}_j\big),
\end{equation}
and the limiting joint-CQ-number at $c$ associated with
$(\mathcal{A},\wt{\mcA},\mathcal{B},\wt{\mcB})$ is
\begin{equation}
\label{e:ljCQn}
\overline\theta=\overline\theta\big(\mathcal{A}, \wt{\mcA},
\mathcal{B},\wt{\mcB}\big)
:=\lim_{\dd\dn0}\theta_\dd\big(\mathcal{A},\wt{\mcA},\mathcal{B},\wt{\mcB}\big).
\end{equation}
\end{definition}
The CQ-number is obviously an instance of the joint-CQ-number when $I$ and $J$ are singletons. 
When the arguments are clear from the context we will simply write $\theta_\dd$ and
$\overline\theta$. 

Using Proposition~\ref{p:elementary}\ref{p:ele-vi}, we see that, for every $x\in X$,
\begin{equation}
\label{e:120405b}
\theta_\dd\big(A,\wt{A},B,\wt{B}\big)\text{~at $c$}
\quad = \quad
\theta_\dd\big(A-x,\wt{A}-x,B-x,\wt{B}-x\big)\text{~at $c-x$.}
\end{equation}

 The CQ-number is based on the behavior of the restricted proximal normal
cone in a neighborhood of a given point.  A related
notion is that of the exact CQ-number, defined next, which is based on the 
restricted normal cone at the point instead of
nearby restricted proximal normal cones.  In both instances, the important case to 
consider is when $c\in A\cap B$ (or when $c\in A_i\cap B_j$ in the joint-CQ case).

\begin{definition}[exact CQ-number and exact joint-CQ-number]
\label{d:exactCQn}
Let $c\in X$.
\begin{enumerate}
\item
Let $A$, $\wt{A}$, $B$ and $\wt{B}$ be nonempty subsets of $X$. The
\emph{exact CQ-number} at $c$ associated with $(A,\wt{A},B,\wt{B})$
is 
\begin{equation}
\label{e:0217a}
\overline{\alpha} :=
\overline{\alpha}\big(A,\wt{A},B,\wt{B}\big) :=
\sup\mmenge{\scal{u}{v}}{u\in\nc{A}{\wt{B}}(c),v\in-\nc{B}{\wt{A}}(c),\|u\|\leq
1, \|v\|\leq 1}.
\end{equation}
where we define  $\overline{\alpha} =\minf$ in the case that $c\notin A\cap B$ which is consistent with 
the convention $\sup\varnothing=\minf$.
\item
Let $\mathcal{A} := (A_i)_{i\in I}$, $\wt{\mcA} := (\wt{A}_i)_{i\in
I}$, $\mathcal{B} := (B_j)_{j\in J}$ and $\wt{\mcB} :=
(\wt{B}_j)_{j\in J}$ be nontrivial collections of nonempty subsets
of $X$. The \emph{exact joint-CQ-number} at $c$ associated with
$(\mathcal{A},\mathcal{B},\wt{\mcA},\wt{\mcB})$ is
\begin{equation}
\label{e:0217b}
\overline{\alpha} := \overline{\alpha}(\mcA,\wt{\mcA},\mcB,\wt{\mcB}) :=
\sup_{(i,j)\in I\times
J}\overline{\alpha}(A_i,\wt{A}_i,B_j,\wt{B}_j).
\end{equation}
\end{enumerate}
\end{definition}

The next result, which we quote from \cite[Theorem 7.8]{BLPW12a}, establishes 
relationships between the condition numbers defined above.
\begin{theorem}\label{t:CQ1}
Let $\mathcal{A} := (A_i)_{i\in I}$, $\wt{\mcA} := (\wt{A}_i)_{i\in
I}$, $\mathcal{B} := (B_j)_{j\in J}$ and $\wt{\mcB} :=
(\wt{B}_j)_{j\in J}$ be nontrivial collections of nonempty subsets
of $X$. Set $A := \bigcup_{i\in I} A_i$ and $B := \bigcup_{j\in
J}B_j$, and suppose that $c\in A\cap B$. Denote the exact
joint-CQ-number at $c$ associated with
$(\mcA,\wt{\mcA},\mcB,\wt{\mcB})$ by $\overline{\alpha}$, 
the joint-CQ-number at $c$ associated with
$(\mcA,\wt{\mcA},\mcB,\wt{\mcB})$ and $\delta>0$ by $\theta_\dd$, 
and the limiting joint-CQ-number at $c$
associated with $(\mcA,\wt{\mcA},\mcB,\wt{\mcB})$ by
$\overline{\theta}$. Then the following hold:
\begin{enumerate}
\item
\label{t:CQ1i} If $\overline\alpha<1$, then the
$(\mcA,\wt{\mcA},\mcB,\wt{\mcB})$-CQ condition holds at $c$.
\item
\label{t:CQ1ii}
$\overline{\alpha}\leq\theta_\delta$.
\item
\label{t:CQ1iii}
$\overline{\alpha}\leq\overline{\theta}$.
\end{enumerate}
If in addition 
$I$ and $J$ are finite,
then the following hold:
\begin{enumerate}[resume]
\item
\label{t:CQ1iv}
$\overline{\alpha}=\overline{\theta}$.
\item
\label{t:CQ1v}
The $(\mcA,\wt{\mcA},\mcB,\wt{\mcB})$-joint-CQ condition holds at $c$
if and only if
$\overline{\alpha}=\overline{\theta}<1$.
\end{enumerate}
\end{theorem}

The CQ-number is related to the angle of intersection of the sets.  
The case of linear subspaces underscores the subtleties of this idea and 
illustrates the connection between the CQ-number and the {\em correct} notion of an 
angle of intersection.  
The \emph{Friedrichs angle}  \cite{Frie37} (or simply the \emph{angle}) between subspaces 
$A$ and $B$ is the
number in $[0,\frac{\pi}{2}]$
whose cosine is given by
\begin{equation}
c(A,B):= \sup\menge{|\scal{a}{b}|}{a\in A\cap(A\cap B)^\perp,b\in B\cap(A\cap B)^\perp, \|a\|\leq 1,\|b\|\leq 1},
\end{equation}
and we set $c(A,B) := c(\pa A,\pa B)$ if $A$ and $B$ are two intersecting
affine subspaces of $X$. 
The following result is a consolidation of 
\cite[Theorem~8.12 and Corollary~8.13]{BLPW12a}.
\begin{theorem}[CQ-number of two (affine) subspaces and Friedrichs angle]
\label{t:CQn=c}
Let $A$ and $B$ be affine subspaces of $X$, and let $\dd>0$.
Then
\begin{equation}\label{e:Fried}
\theta_\dd(A,A,B,B)=\theta_\dd(A,X,B,B)=\theta_\dd(A,A,B,X)=c(A,B)<1,
\end{equation}
where the CQ-number at $0$ is defined as in \eqref{e:CQn}.

Moreover, if $A$ and $B$ are {\em affine} subspaces of $X$ with
$c\in A\cap B$, and $\dd>0$, then \eqref{e:Fried} holds at $c$.  
\end{theorem}
An easy consequence of Theorem \ref{t:CQn=c} is the case of two distinct lines through the 
origin for which the CQ-number is simply the 
cosine of the angle between them (\cite[Proposition 7.3]{BLPW12a}). 
  \begin{corollary}[two distinct lines through the origin]
\label{p:CQn2l}
Suppose that $w_a$ and $w_b$ are two vectors in $X$ such that
$\|w_a\|=\|w_b\|=1$.
Let $A :=\RR w_a$, $B:= \RR w_b$, and $\dd>0$.
Assume that $A\cap B = \{0\}$.
Then the CQ-number at $0$ is
\begin{equation}
\theta_\dd(A,A,B,B)=\theta_\dd(A,X,B,B)=\theta_\dd(A,A,B,X)=c(A,B)=|\scal{w_a}{w_b}|<1.
\end{equation}
\end{corollary}

Convergence of MAP requires also a certain regularity on neighborhoods of the corresponding fixed points.  
For this we used a notion of 
regularity of the sets that is an adaptation to restricted normal cones
of type of regularity introduced in \cite{LLM}.
\begin{definition}[(joint-) regularity and (joint-) superregularity]\label{d:reg}
Let $A$ and $B$ be nonempty subsets of $X$, 
let $\mcB := (B_j)_{j\in J}$ be a nontrivial collection of nonempty subsets of $X$,
and let $c\in X$.
\begin{enumerate}
\item
We say that $B$ is \emph{$(A,\ve,\dd)$-regular} at $c\in X$ if
$\ve\geq 0$, $\dd>0$, and
\begin{equation}
\label{e:dreg} \left.
\begin{array}{c}
(y,b)\in B\times B,\\
\|y-c\| \leq \dd,\|b-c\|\leq \dd,\\
u\in \pn{B}{A}(b)
\end{array}
\right\}\quad\Rightarrow\quad \scal{u}{y-b}\leq
\ve\|u\|\cdot\|y-b\|.
\end{equation}
If $B$ is $(X,\ve,\dd)$-regular at $c$, then
we also simply speak of $(\ve,\dd)$-regularity.
\item
The set $B$ is called $A$-\emph{superregular} at $c\in X$ if for
every $\ve>0$ there exists $\dd>0$ such that $B$ is
$(A,\ve,\dd)$-regular at $c$.
Again, if $B$ is $X$-superregular at $c$, then we also say
that $B$ is superregular at $c$.
\item We say that $\mcB$ is $(A,\ve,\dd)$-joint-regular at $c$
if $\ve\geq 0$, $\dd>0$, and for every $j\in J$,
$B_j$ is $(A,\ve,\dd)$-regular at $c$.  
\item
The collection $\mcB$ is $A$-joint-superregular at $c$
if for every $j\in J$, $B_j$ is $A$-superregular at $c$.
We omit the prefix $A$ if $A=X$.
\end{enumerate}
\end{definition}

Joint-(super)regularity can be easily checked by any of the following conditions.
\begin{proposition}
\label{p:jsreg}
Let $\mcA := (A_j)_{j\in J}$ and $\mcB := (B_j)_{j\in J}$
be nontrivial collections of nonempty subsets of $X$,
let $c\in X$, let $(\ve_j)_{j\in J}$ be a collection in $\RP$,
and let $(\dd_j)_{j\in J}$ be a collection in $\left]0,\pinf\right]$.
Set $A := \bigcap_{j\in J}A_j$,
$\ve := \sup_{j\in J}\ve_j$,
and $\dd := \inf_{j\in J}\dd_j$.
Then the following hold:
\begin{enumerate}
\item
\label{p:jsreg-i}
If $\dd>0$ and $(\forall j\in J)$ $B_j$ is $(A_j,\ve_j,\dd_j)$-regular
at $c$, then $\mcB$ is $(A,\ve,\dd)$-joint-regular at $c$.
\item
\label{p:jsreg-ii}
If $J$ is finite and
$(\forall j\in J)$ $B_j$ is $(A_j,\ve_j,\dd_j)$-regular at $c$,
then $\mcB$ is $(A,\ve,\dd)$-joint-regular at $c$.
\item
\label{p:jsreg-iii}
If $J$ is finite and
$(\forall j\in J)$ $B_j$ is $A_j$-superregular at $c$,
then $\mcB$ is $A$-joint-superregular at $c$.
\end{enumerate}

If in addition $\mcB := (B_j)_{j\in J}$ is a nontrivial collection of nonempty
{\em convex} subsets of $X$ then, for ${A}\subseteq X$,
$\mcB$ is $(0,\pinf)$-joint-regular, $(A,0,\pinf)$-joint-regular,
joint-superregular, and $A$-joint-superregular at $c\in X$.

\end{proposition}

The framework of restricted normal cones allows for a great deal of 
flexibility in how one decomposes problems.  Whatever the chosen
decomposition, the following properties will be required.
\boxedeqn{
\label{e:MAPsettings}
\left\{
\begin{aligned}
&\text{$\mcA := (A_i)_{i\in I}$ and $\mcB := (B_j)_{j\in J}$ are
nontrivial collections}\\
&\quad\text{of nonempty closed subsets of $X$;}\\
&A :=\bigcup_{i\in I} A_i \text{~and~}
B:= \bigcup_{j\in J} B_j\text{~are closed;}\\
&c\in A\cap B; \\
&\text{$\wt{\mcA} := (\wt{A}_i)_{i\in I}$ and
$\wt{\mcB} := (\wt{B}_j)_{j\in J}$ are collections}\\
&\quad\text{of nonempty subsets of $X$ such that }\\
&\qquad (\forall i\in I)\;\;P_{A_i}\big((\bd B)\smallsetminus
A\big)\subseteq\wt{A}_i,\\
&\qquad (\forall j\in J)\;\;P_{B_j}\big((\bd A)\smallsetminus
B\big)\subseteq\wt{B}_j;\\
&\wt{A} :=\bigcup_{i\in I} \wt{A}_i \text{~and~}
\wt{B}:= \bigcup_{j\in J} \wt{B}_j.
\end{aligned}
\right.
}

With the above assumptions one can establish rates of convergence for the MAP algorithms.
\begin{theorem}[convergence rate, Corollary 10.8 of \cite{BLPW12a}]
\label{p:jdb}
Assume that \eqref{e:MAPsettings} holds and
that there exists $\dd>0$ such that
\begin{enumerate}
\item $\mcA$ is $(\wt{B},0,3\dd)$-joint-regular at $c$;
\item $\mcB$ is $(\wt{A},0,3\dd)$-joint-regular at $c$; and
\item $\theta<1$, where
$\theta:=\theta_{3\dd}$ is
the joint-CQ-number at $c$ associated with
$(\mcA,\wt{\mcA},\mcB,\wt{\mcB})$
(see Definition~\ref{d:CQn}).
\end{enumerate}
Suppose also that the starting point of the MAP $b_{-1}$ satisfies
$\|b_{-1}-c\|\leq\frac{(1-\theta)\dd}{6(2-\theta)}$.
Then $(a_k)_\kkk$ and $(b_k)_\kkk$ converge linearly
to some point in $\bar{c}\in A\cap B$ with $\|\bar{c}-c\|\leq\dd$
and rate $\theta^2$; in fact,
\begin{equation}
(\forall k\geq1)\quad
\max\big\{\|a_k-\bar{c}\|,\|b_k-\bar{c}\|\big\}\leq
\frac{\dd}{2-\theta}\big(\theta^2\big)^{k-1}.
\end{equation}
\end{theorem}


\section{Sparse feasibility with an affine constraint}
\label{s:russell}


We now move to the application of feasibility with a sparsity set and an affine subspace, 
problem \eqref{e:sparse feas}.  Our main result on the convergence of MAP 
is given in Theorem \eqref{t:mainsparse}.  Along the way we develop explicit 
representations of the projections, normal cones, and tangent cones to the 
sparsity set \eqref{e:AB} and motivate our decomposition of the problem.   

\subsection*{Properties of sparsity sets}
\begin{lemma}
\label{l:0314a}
Let $x$ and $y$ be in $\RR^\dm$, and let $\lambda\in\RR$.
Then the following hold:
\begin{enumerate}
\item
\label{l:0314a1}
$\supp(x) = \lspan\menge{e_i}{i\in I(x)}$ and
$\|x\|_0 = \card(I(x))=\dim \supp(x)$.
\item
\label{l:0314a2}
$x\in\supp(y)$
$\Leftrightarrow$
$I(x)\subseteq I(y)$
$\Leftrightarrow$
$\supp(x)\subseteq \supp(y)$
$\Rightarrow$
$\|x\|_0\leq\|y\|_0$.
\item
\label{l:0314a2+}
$I(x+y)\subseteq I(x)\cup I(y)$ and
$\displaystyle I(\lambda x) = \begin{cases}I(x), &\text{if $\lambda\neq 0$;}\\
\varnothing, &\text{otherwise.}\end{cases}$
\item
\label{l:0314a3}
$I((1-\lambda)x+\lambda y)\subseteq I(x)\cup I(y)$.
\item
\label{l:0314a3+-}
$\supp(\lambda x)=\lambda\supp(x)$ and $\|\lambda
x\|_0=|\sgn(\lambda)|\cdot\|x\|_0$.
\item
\label{l:0314a3+}
$\supp(x+y)\subseteq\supp(x)+\supp(y)$ and $\|x+y\|_0\leq \|x\|_0 + \|y\|_0$.
\item
\label{l:0314a4-}
If $\supp(x)\subseteq\supp(y)$ and $z\in\supp(y)$,
then there exist $u$ and $v$ in $\RR^\dm$ such that
$z=u+v$, $u\in\supp(x)$ and $\|v\|_0\leq \|y\|_0-\|x\|_0$.
\item
\label{l:0314a4}
Let $\delta\in\left]0,\min\menge{|x_i|}{i\in I(x)}\right[$ and
$y\in x+[-\dd,+\dd]^\dm$,
then $\supp(x)\subseteq\supp(y)$.
\item
\label{l:0314a5}
If $I(x)\nsubseteq I(y)$ and $I(y)\nsubseteq I(x)$, then
\begin{equation}
\|x+y\|^2\geq\min_{i\in I(x)\smallsetminus I(y)}|x_i|^2
+\min_{j\in I(y)\smallsetminus I(x)}|y_j|^2
\geq\min_{i\in I(x)}|x_i|^2+\min_{j\in I(y)}|y_j|^2.
\end{equation}
\item
\label{l:0314a6}
$\|\cdot\|_0$ is lower semicontinuous.
\end{enumerate}
\end{lemma}
\begin{proof}
\ref{l:0314a1}--\ref{l:0314a3+-}:
These follow readily from the definitions.

\ref{l:0314a3+}:
By \ref{l:0314a2}, $I(x+y)\subseteq I(x)\cup I(y)$.
Hence $\supp(x+y)\subseteq\supp(x)+\supp(y)$; on the other hand,
taking cardinality and using \ref{l:0314a1} yields
$\|x+y\|_0\leq\|x\|_0+\|y\|_0$.

\ref{l:0314a4-}:
By \ref{l:0314a2}, we have $I(x)\subseteq I(y)$.
Write $I(y)=I(x)\dotcup J$ as disjoint union, where
$J=I(y)\smallsetminus I(x)$, and note that
that $\card(J)=\card(I(y))-\card(I(x))= \|y\|_0-\|x\|_0$.
Then $\supp(y) = \supp(x)\oplus \lspan\menge{e_i}{i\in J}$.
Now since $z\in\supp(y)$, we can write $z=u+v$,
where $u\in\supp(x)$ and $v\in \lspan\menge{e_i}{i\in J}$
and $\|v\|_0\leq\card(J)=\|y\|_0-\|x\|_0$.

\ref{l:0314a4}:
If $i\in I(x)$, then $|y_i|\geq |x_i|-|x_i-y_i|>\dd-|x_i-y_i|\geq 0$ and
hence $y_i\neq 0$. It follows that $I(x)\subseteq I(y)$. Now apply \ref{l:0314a2}.

\ref{l:0314a5}:
Let $i_0\in I(x)\smallsetminus I(y)$ and $j_0\in I(y)\smallsetminus I(x)$.
Then $y_{i_0}=0$ and $x_{j_0}=0$, and hence
\begin{subequations}
\begin{align}
\|x+y\|^2&\geq|x_{i_0}+y_{i_0}|^2+|x_{j_0}+y_{j_0}|^2\\
&\geq\min_{i\in I(x)\smallsetminus I(y)}|x_i|^2+
\min_{j\in I(y)\smallsetminus I(x)}|y_j|^2\\
&\geq \min_{i\in I(x)}|x_i|^2+\min_{j\in I(y)}|y_j|^2,
\end{align}
\end{subequations}
as claimed.

\ref{l:0314a6}:
Indeed, borrowing the notation below,
we see that $\menge{z\in X}{\|z\|_0\leq \rho}=
\bigcup_{J\in\mcJ_r} A_J$, where $r = \lfloor \rho\rfloor$,
is closed as a union of finitely many (closed) linear subspaces.
\end{proof}

In order to apply Theorem \ref{p:jdb} to MAP for solving \eqref{e:sparse feas} we must choose 
a suitable decomposition, $\mcA$ and $\mcB$, and restrictions, $\wt{\mcA}$ and $\wt{\mcB}$, 
and verify the assumptions of the theorem.   We now abbreviate
\begin{subequations}
\label{e:russell}
\begin{equation}
\mcJ :=  2^{\{1,2,\ldots,\dm\}}
\quad\text{and}\quad
\mcJ_s := \mcJ(s) := \menge{J\in\mcJ}{\card(J)=s}
\end{equation}
and set
\begin{equation}
(\forall J\in\mcJ)\quad
A_J := \lspan\menge{e_j}{j\in J}.
\end{equation}
Define the collections
\begin{equation}
\mcA := \wt{\mcA}:= (A_J)_{J\in\mcJ_s}
\quad\text{and}\quad
\mcB:=\wt{\mcB} := (B).
\end{equation}
Clearly,
\begin{equation}
A:=\wt{A}:=\bigcup_{J\in \mcJ_s} A_J = \menge{x\in \RR^\dm}{\|x\|_0\leq s} 
\quad\text{and}\quad
B= \wt{B} := \menge{x\in X}{Mx=p}.
\end{equation}
\end{subequations}


The proofs of the following two results are elementary and thus omitted.

\begin{proposition}[properties of $A_J$]\label{p:AJbasic}
Let $J$, $J_1$, and $J_2$ be in $\mcJ$, and let
$x\in X$.
Then the following hold:
\begin{enumerate}
\item\label{p:AJbasic-i}
$A_{J_1}\cup A_{J_2}\subseteq A_{J_1\cup J_2}=\lspan(A_{J_1}\cup A_{J_2})$.
\item\label{p:AJbasic-ii}
$J_1\subseteq J_2$ $\Leftrightarrow$ $A_{J_1}\subseteq A_{J_2}$.
\item\label{p:AJbasic-ii+}
$x\in A_{I(x)}=\supp(x)$.
\item\label{p:AJbasic-iii}
$I(x)\subseteq J$ $\Leftrightarrow$ $x\in A_J$.
\item\label{p:AJbasic-iv}
$I(x)\cap J=\emp$ $\Leftrightarrow$ $x\in A_J^\perp$.
\item $s\leq \dm-1$ $\Leftrightarrow$ $\inte A=\varnothing$. 
\end{enumerate}
\end{proposition}

\begin{proposition}
\label{p:projAJ}
Let $J\in\mcJ$, let $x=(x_1,\ldots,x_\dm)\in X$,
and set $y := P_{A_J}x$.
Then
\begin{equation}\label{e:projAJ}
(\forall i\in \{1,\ldots,\dm\})\quad
y_i = \begin{cases}
x_i, &\text{if $i\in J$;}\\
0, &\text{if $i\notin J$,}
\end{cases}
\end{equation}
and
\begin{equation}\label{e:distAJ}
d^2_{A_J}(x)=\sum_{j\in\{1,\ldots,n\}\smallsetminus J}|x_j|^2
=\sum_{j\in I(x)\smallsetminus J}|x_j|^2.
\end{equation}
\end{proposition}

The following technical result will be useful later.

\begin{lemma}
\label{l:120412c}
Let $c\in A$, and assume that $s\leq\dm-1$.
Then
\begin{equation}
\min\menge{d_{A_J}(c)}{c\not\in A_J,J\in \mcJ_s}=\min\menge{|c_j|}{j\in I(c)}.
\end{equation}
\end{lemma}
\begin{proof}
First, let $J\in\mcJ_s$ such that $c\not\in A_J$
$\Leftrightarrow$ $I(c)\not\subseteq J$
by Proposition~\ref{p:AJbasic}\ref{p:AJbasic-iii}.
So $I(c)\smallsetminus J\neq\emp$.
By \eqref{e:distAJ},
$d^2_{A_J}(c)=\sum_{j\in I(c)\smallsetminus J}|c_j|^2\geq
\min\menge{|c_j|^2}{j\in I(c)}$.
Hence
\begin{equation}
\min\menge{d_{A_J}(c)}{c\not\in A_J,J\in \mcJ_s}
\geq\min\menge{|c_j|}{j\in I(c)}.
\end{equation}
Since $1\leq 1 + s-\|c\|_0 \leq
\dm-\|c\|_0=\card(\{1,\ldots,\dm\}\smallsetminus I(c))$,
there exists a nonempty subset $K$ of $\{1,\ldots,\dm\}\smallsetminus
I(c)$ with $\card(K) = s-\|c\|_0+1$.
Let $j\in I(c)$ such that $|c_j| = \min_{i\in I(c)}|c_i|$ and set
\begin{equation}
\label{e:120411b}
J := (I(c)\smallsetminus\{j\})\cup K.
\end{equation}
Then $c\notin A_J$ and
$\card(J) = \card(I(c))-1+\card(K)=\|c\|_0 - 1+s-\|c\|_0+1=s$.
Hence $J\in\mcJ_s$.
Because $I(c)\smallsetminus J=\{j\}$, it follows again from
\eqref{e:distAJ} that
$d^2_{A_J}(c) = \sum_{i\in I(c)\smallsetminus J} |c_i|^2 = |c_j|^2$.
Therefore $d_{A_J}(c) = |c_j| = \min_{i\in I(c)}|c_i|$, which yields
the inequality complementary to \eqref{e:120411b}.
\end{proof}

Now let $x=(x_1,...,x_\dm)\in X$, and set
\begin{equation}
\mcC_s(x):=
\menge{J\in\mcJ_s}{\min_{j\in J}|x_j|\geq \max_{k\not\in J}|x_{k}|};
\end{equation}
in other words,
$J\in \mcC_s(y)$ if and only if $J$ contains $s$ indices to the $s$ largest
coordinates of $x$ in absolute value.

The proof of the next result is straightforward.

\begin{lemma}\label{l:0401a}
Let $x=(x_1,\ldots,x_\dm)\in X$ such that
$\|x\|_0=\card(I(x))\geq s$,
and let $J\in\mcC_s(x)$.
Then $J\subseteq I(x)$ and
$\min_{j\in J}|x_j|\geq \min_{j\in I(x)}|x_j| >0$.
If $\|x\|_0=\card(I(x))=s$, then $\mcC_s(x)=\{I(x)\}$.
\end{lemma}

\subsection*{Projections}
The decomposition of the sparsity set defined by \eqref{e:russell} yields a natural expression 
for the projection onto this set.
\begin{proposition}[Projection onto $A$ and its inverse]\label{p:projAU}
Let $x=(x_1,\ldots,x_\dm)\in X$, and define 
$A:=\menge{x\in X}{\|x\|_0\leq s}$.
Then the following hold:
\begin{enumerate}
\item\label{p:projAU-i}
The distance from $x$ to $A$ is solely determined by $\mcC_s(x)$;
more precisely,
\begin{equation}\label{e:distAU}
(\forall J\in\mcJ_s)\quad
d_{A_J}(x)  \begin{cases}
= d_A(x), &\text{if $J\in\mcC_s(x)$;}\\
> d_A(x), &\text{if $J\notin \mcC_s(x)$.}
\end{cases}
\end{equation}
\item\label{p:projAU-ii}
The projection of $x$ on $A$ is solely determined by $\mcC_s(x)$;
more precisely,
\begin{equation}\label{e:projAU}
P_A(x)=\bigcup_{J\in \mcC_s(x)}P_{A_J}(x)
=\bigcup_{J\in \mcC_s(x)}\MMenge{y=(y_1,\ldots,y_\dm)\in X}
{\begin{aligned}
(\forall j\in \{1,\ldots,\dm\})\; y_j=\begin{cases}
x_j,&\text{if $j\in J$;}\\
0, &\text{if $j\notin J$.}
\end{cases}
\end{aligned}}
\end{equation}
\item\label{p:projAU-ii+}
$(\forall y\in P_{A}(x))$  $\|y\|_0=\min\{\|x\|_0,s\}$.
\item
\label{p:projAU-ii++} 
If $x\not\in A$, then $(\forall y\in P_{A}(x))$
$I(y)\in\mcC_s(x)$ and $\|y\|_0=s$.
\item\label{p:projAU-iii}
If $a\in A$ and $\|a\|_0=s$, then
\begin{equation}
\label{e:120410d}
P^{-1}_A(a)=\mmmenge{y=(y_1,\ldots,y_\dm)\in X}
{\begin{aligned}
&(\forall j\in I(a))\; y_j=a_j\\
&\max_{k\notin I(a)} |y_{k}|\leq \min_{j\in I(a)} |a_j|.
\end{aligned}
}
\end{equation}
\item\label{p:projAU-iv}
If $a\in A$ and $\|a\|_0<s$, then $P^{-1}_A(a)=a$.
\end{enumerate}
\end{proposition}
\begin{proof}
The following observation will be useful.
If $J\in\mcJ_s$,
$j\in J$, and $k\notin J$,
then $K := (J\smallsetminus\{j\})\cup \{k\} \in \mcJ_s$ and
\eqref{e:distAJ} implies
\begin{subequations}
\label{e:0410a}
\begin{align}
d_{A_{K}}^2(x)
&= \sum_{l\notin K} |x_l|^2
= \|x\|^2 - \sum_{l\in K}|x_l|^2
= \|x\|^2 - \sum_{l\in J\cap K}|x_l|^2 - |x_{k}|^2\\
&= \|x\|^2 - \sum_{l\in J\cap K}|x_j|^2 - |x_{j}|^2 +
\big(|x_{j}|^2-|x_{k}|^2\big)\\
&= \|x\|^2 - \sum_{l\in J}|x_l|^2 +
\big(|x_{j}|^2-|x_{k}|^2\big)
= \sum_{l\notin J} |x_l|^2 + \big(|x_{j}|^2-|x_{k}|^2\big)\\
&= d_{A_{J}}^2(x) + |x_{j}|^2-|x_{k}|^2.
\end{align}
\end{subequations}

\ref{p:projAU-i}:
It is clear that
\begin{equation}
d_A(x) = \min\menge{d_{A_J}(x)}{J\in \mcJ_s}.
\end{equation}
Let $K\in \mcJ_s$ and assume that $K\notin \mcC_s(x)$.
Then there exists $j$ and $k$ in $\{1,\ldots,\dm\}$ such that
$k\in K$, $j\notin K$, and $|x_k|<|x_j|$.
Now define $J = (K\smallsetminus\{k\})\cup\{j\}$.
Then $J\in \mcJ_s$ and
\begin{equation}
d_{A_{K}}^2(x) = d_{A_{J}}^2(x) + |x_{j}|^2-|x_{k}|^2
>  d_{A_{J}}^2(x)
\end{equation}
by \eqref{e:0410a}.
It follows that index sets in $\mcJ_s\smallsetminus \mcC_s(x)$ do not
contribute to the computation of $d_A(x)$.

Now assume that $J$ and $K$ both belong to $\mcC_s(x)$ and
that $J\neq K$. Then $\card(J\smallsetminus K)=\card(K\smallsetminus J)$.
Take $j\in J\smallsetminus K$ and $k\in K\smallsetminus J$.
Since $j\in J\in\mcC_s(x)$ and $k\notin J$, we have
$|x_j|\geq |x_k|$.
On the other hand, since $k\in K\in \mcC_s(x)$ and $j\notin K$,
we also have $|x_k|\geq |x_j|$. Altogether, $|x_j|=|x_k|$.
Thus
\begin{subequations}
\begin{align}
d_{A_J}^2(x)&= \|x\|^2 - \sum_{l\in J}|x_l|^2
= \|x\|^2 - \sum_{l\in J\cap K}|x_l|^2 - \sum_{l\in J\smallsetminus
K}|x_l|^2 \\
&= \|x\|^2 - \sum_{l\in K\cap J}|x_l|^2 - \sum_{l\in K\smallsetminus J}|x_l|^2
= \|x\|^2 - \sum_{l\in K}|x_l|^2
= d_{A_K}^2(x).
\end{align}
\end{subequations}
This completes the proof of \eqref{e:distAU}.

\ref{p:projAU-ii}: This follows from \eqref{e:distAU} and \eqref{e:projAJ}.

\ref{p:projAU-ii+}:
\emph{Case~1:} $\|x\|_0=\card(I(x))\leq s$.
Then, by definition, $x\in A$. Thus $P_A(x)=x$
and hence $\|P_A(x)\|_0=\|x\|_0=\min\{\|x\|_0,s\}$.

\emph{Case~2:}
$\|x\|_0=\card(I(x))>s$.
Let $J\in\mcC_s(x)$.
Lemma~\ref{l:0401a} implies
$\min_{j\in J} |x_j|>0$.
It follows from \eqref{e:projAU} that
there exists $y=(y_1,\ldots,y_\dm)\in P_A(x)$
such that 
\begin{equation}
  (\forall j\in J)\; |y_j|=|x_j|>0\quad\text{and}\quad
  (\forall j\not\in J)\; y_j=0.
\end{equation}
So 
\begin{equation}
\label{e:120410b}
I(y)=J,
\end{equation}
and hence $\|y\|_0=\card(J)=s=\min\{\card(I(x)),s\}$.

\ref{p:projAU-ii++}:
Let $y\in P_A(x)$. 
Since $x\notin A$, we have $\|x\|_0>s$ and hence
\ref{p:projAU-ii+} implies that $\|y\|_0=s$.
By \eqref{e:projAU}, there exists $J\in\mcC_s(x)$ such that
$I(y)\subseteq J$. But $\card I(y)=s=\card J$, and hence
$I(y)=J$. 

\ref{p:projAU-iii}:
Denote the right-hand side of \eqref{e:120410d} by $R$.
``$\supseteq$'': for every $y\in R$,
we have $I(a)\in\mcC_s(y)$. By \eqref{e:projAU}, $a\in P_A y$.
Hence $y\in P_A^{-1}(a)$. This establishes $R\subseteq P_A^{-1}(a)$.
``$\subseteq$'':
Suppose that $y\in P^{-1}_A(a)$, i.e., $a\in P_A(y)$.
Again by \eqref{e:projAU},
there exists $J\in\mcC_s(y)$ such that
\begin{equation}
\label{e:120410c}
(\forall j\in J)\ a_j=y_j\quad\mbox{and}\quad(\forall j\not\in J)\ a_j=0.
\end{equation}
Since $\|a\|_0=s$, Lemma~\ref{l:0401a} implies that
$J=I(a)$.
Hence, by \eqref{e:120410c},
$(\forall j\in I(a))$ $y_j=a_j$.
On the other hand, by definition of $\mcC_s(y)$, we have
$\min_{j\in J}|y_j| \geq \max_{k\notin J}|y_k|$.
Altogether, $y\in R$.

\ref{p:projAU-iv}: Let $y\in P^{-1}_A a$, i.e., $a\in P_Ay$.
The hypothesis and \ref{p:projAU-ii+} imply
$s>\|a\|_0=\min\{\|y\|_0,s\}$,
Hence $\|y\|_0<s$; therefore, $y\in A$ and so $a=P_Ay=y$.
\end{proof}

\begin{proposition}[projection onto $B$]
{\rm (See \cite[Lemma~4.1]{BK}.)}
\label{p:projAx=b}
Recall that $B = \menge{x\in X}{Mx=p}$.
Then the projection onto $B$ is given by
\begin{equation}
\label{e:projAx=b}
P_B\colon X\to X\colon x\mapsto x - M^\dagger(Mx-p),
\end{equation}
where $M^\dagger$ denotes the Moore-Penrose inverse of $M$.
\end{proposition}


\subsection*{Normal and tangent cones}

\begin{proposition}[proximal normal cone to $A$]\label{p:prox-spr}
\begin{equation}
(\forall a\in A)\quad
\pnX{A}(a)=\begin{cases} (\supp(a))^\bot, &\text{if $\|a\|_0=s$;}\\
\{0\}, &\text{if $\|a\|_0<s$.}
\end{cases}
\end{equation}
\end{proposition}
\begin{proof}
Combine the definition of $\pnX{A}(a)$ with
Proposition~\ref{p:projAU}\ref{p:projAU-iii}\&\ref{p:projAU-iv}.
\end{proof}

The following is a special case of a more general normal cone formulation for 
the set of matrices with rank bounded above by $s$ given in \cite{Luke12b}.
\begin{theorem}[Mordukhovich normal cone to $A$]\label{t:Mnor-spr}
\begin{equation}\label{e:Mnor-spr}
(\forall a\in A)\quad N_A(a) = \menge{u\in \RR^\dm}{\|u\|_0\leq \dm-s}\cap
 \big(\supp(a)\big)^\perp=\bigcup_{I(a)\subseteq J\in\mcJ_s} A_J^\perp.
\end{equation}
Consequently, if $\|a\|_0=s$, then \ $N_A(a)=(\supp(a))^\perp=A_{I(a)}^\perp$.
\end{theorem}
\begin{proof}
Let $a\in A$, and let $\ve\in\left]0,\min\menge{a_j}{j\in I(a)}\right[$.
Let $x=(x_1,\ldots,x_\dm)\in A \cap
\big(a+\left[-\ve,+\ve\right]^\dm\big)$.
Then $\|x\|_0\leq s$ and, by Lemma~\ref{l:0314a}\ref{l:0314a4},
$\supp(a)\subseteq\supp(x)$.
Hence, using Proposition~\ref{p:prox-spr}, we deduce that
\begin{equation}\label{e:0401b}
\pnX{A}(x)=\begin{cases} (\supp(x))^\bot ,& \text{if $\|x\|_0=s$;}\\
\{0\} ,& \text{if $\|x\|_0<s$}
\end{cases}\; \subseteq\; \big(\supp(a)\big)^\bot.
\end{equation}
Note that  if $\|x\|_0=s$, then \eqref{e:0401b}
yields $\dim (\supp(x))^\perp =\dm-s$;
in either case,
\begin{equation}
\label{e:120410e}
\big(\forall u\in \pnX{A}(x)\big)\quad
\|u\|_0\leq \dm-s.
\end{equation}

Let $u\in X$. We assume first that $u\in N_A(a)$.
Then there exist sequences $(x_k)_{k\in\NN}$ in
$A\cap \big(a+\left[-\ve,+\ve\right]^\dm\big)$ and
$(u_k)_{k\in\NN}$ in $X$ such that
$x_k\to a$,
$u_k\to u$, and
$(\forall k\in\NN)$ $u_k\in\pnX{A}(x_k)$.
It follows from \eqref{e:0401b}, \eqref{e:120410e},
and Lemma~\ref{l:0314a}\ref{l:0314a6} that $u\in(\supp(a))^\perp$
and $\|u\|_0\leq\dm-s$.
Thus
\begin{equation}
N_A(a) \subseteq \menge{u\in \RR^\dm}{\|u\|_0\leq \dm-s}\cap
 \big(\supp(a)\big)^\perp.
\end{equation}
We now assume that
$u\in(\supp(a))^\bot$ and $\|u\|_0\leq\dm-s$.
Since $u\in(\supp(a))^\bot$, we have $I(a)\cap I(u)=\emp$ and hence
$I(a)\subset \{1,2,\ldots,\dm\}\smallsetminus I(u)$.
Since $a\in A$ and $\card I(u)=\|u\|_0\leq \dm-s$, we have
$\card{I(a)}\leq s \leq \card(\{1,2,\ldots,\dm\}\smallsetminus I(u))$.
Let $J\in\mcJ_s$ such that
$I(a)\subseteq J\subseteq \{1,2,\ldots,\dm\}\smallsetminus I(u)$.
By Proposition~\ref{p:AJbasic}\ref{p:AJbasic-iv},
$u\in A_J^\perp$.
We have established that
\begin{equation}
\menge{u\in \RR^\dm}{\|u\|_0\leq \dm-s}\cap
 \big(\supp(a)\big)^\perp\subseteq
\bigcup_{I(a)\subseteq J\in\mcJ_s} A_J^\perp.
\end{equation}
Finally, assume that
$u\in A_J^\perp$, where $\card J=s$ and $I(a)\subseteq J$.
Set
\begin{equation}
(\forall\ve\in\RPP)(\forall j\in\{1,2,\ldots,\dm\})\quad
x_{\ve,j}:=\begin{cases}
a_j,& \text{if $j\in I(a)$;}\\
\ve,& \text{if $j\in J\smallsetminus I(a)$;}\\
0& \text{otherwise}.
\end{cases}
\end{equation}
This defines a bounded net $(x_\ve)_{\ve\in\zeroun}$ in $X$
with $x_\ve\to a$ as $\ve\to 0$.
Note that $(\forall \ve\in\zeroun)$
$I(x_\ve)=J$; hence, $x_\ve\in\supp(x_\ve)=A_J\subseteq A$
and, by Proposition~\ref{p:prox-spr},
$u\in A_J^\perp = (\supp(x_\ve))^\perp =\pnX{A}(x_\ve)$.
Thus $u\in N_A(a)$. We have established the inclusion
\begin{equation}
\bigcup_{I(a)\subseteq J\in\mcJ_s} A_J^\perp\subseteq N_A(a).
\end{equation}
This completes the proof of \eqref{e:Mnor-spr}.

Finally, if $\|a\|_0=s$, then $\card I(a)=s$
and the only choice for $J$ in \eqref{e:Mnor-spr} is $I(a)$.
\end{proof}

We now turn to the classical tangent cone of $A$.


\begin{definition}[tangent cone]
Let $C$ be a nonempty subset of $X$, and let $c\in C$.
Then a vector $v\in X$ belongs to the \emph{tangent cone} to $C$ at $c$,
denoted $T_C(c)$, if there exist sequences $(x_k)_{k\in\NN}$ in $C$ and
$(t_k)_{k\in\NN}$ in $\RPP$ such that 
$x_k\to c$, $t_k\to 0$, and $(x_k-c)/t_k\to v$.
\end{definition}

The proof of the following result is elementary and hence omitted.

\begin{lemma}
\label{l:0406a}
Let $C$ be a nonempty subset of $X$,
let $c\in C$, and assume that $(Y_k)_{k\in K}$ a finite collection
of affine subspaces such that $y\in \bigcap_{k\in K} Y_k\subseteq Y :=
\bigcup_{k\in K} Y_k$.
Then the following hold:
\begin{enumerate}
\item
\label{l:0406a-i}
$(\forall\rho\in\RPP)$ $T_C(c) = T_{(C\cap\ball{c}{\rho})}(c)$.
\item
\label{l:0406a-ii}
$T_Y(y)=\bigcup_{k\in K}\pa(Y_k)$.
\item
\label{l:0406a-iii}
If each $Y_k$ is a linear subspace, then $T_Y(y)= Y$.
\end{enumerate}
\end{lemma}

\begin{lemma}
\label{l:0406c}
Let $a=(a_1,\ldots,a_\dm)\in A$ and suppose that
$\disp 0<\rho\leq\min_{j\in I(a)}|a_j|$. Then
\begin{equation}\label{e:0406c}
\ball{a}{\rho}\cap A = \ball{a}{\rho} \cap \bigcup_{I(a)\subseteq J\in \mcJ_s}A_J.
\end{equation}
\end{lemma}
\begin{proof}
The inclusion ``$\supseteq$'' is clear.
To prove ``$\subseteq$'', let $x\in A\cap\ball{a}{\rho}$.
If $I(x)\nsubseteq I(a)$ and $I(a)\nsubseteq I(x)$,
then Lemma~\ref{l:0314a}\ref{l:0314a5} implies
$\rho^2\geq\|x-a\|^2\geq \min_{i\in I(x)}|x_i|^2+\min_{j\in I(a)}|a_j|^2
>\rho^2$, which is absurd.
Therefore, $I(x)\subseteq I(a)$ or $I(a)\subseteq I(x)$.
Furthermore, there exists $J\in\mcJ_s$ such that
$I(a)\subseteq I(a)\cup I(x)\subseteq J$.
By Proposition~\ref{p:AJbasic}\ref{p:AJbasic-iii}, $x\in A_J$.
This completes the proof.
\end{proof}

\begin{corollary}
Let $a\in A$. 
If $s=\dm$, then $A$ is superregular at $a$;
otherwise, 
$A$ is superregular at $a$ $\Leftrightarrow$ $\|a\|_0=s$. 
\end{corollary}
\begin{proof}
Since $A=X$ if $s=\dm$, the first statement is clear.
We now consider two cases. 
\emph{Case~1:} $\|a\|_0\leq s-1$.
By \eqref{e:Mnor-spr},
\begin{equation}
N_A(a) = \bigcup_{I(a)\subseteq J\in\mcJ_s} A_J^\perp.
\end{equation}
Since $\card I(a)<s$, $N_A(a)$ is therefore
the finite union of two or more different linear subspaces of $X$ all
of the same dimension $\dm-s$. Hence $N_A(a)$ cannot be convex.
On the other hand, $\fnc{A}(a)$ is always convex.
Altogether, $\fnc{A}(a)\neq N_A(a)$.
Thus, by \cite[Definition~6.4]{Rock98},
$A$ is not Clarke regular at $a$.
Hence \cite[Corollary~4.5]{LLM} implies that $A$ is not superregular at $a$.

\emph{Case~2:} $\|a\|_0= s$.
Let $\rho$ be as in Lemma~\ref{l:0406c}.
Then Lemma~\ref{l:0406c} implies that
\begin{equation}
\ball{a}{\rho}\cap A=\ball{c}{\rho}\cap A_{I(a)}
\end{equation}
is convex because it is the intersection of a ball and a linear
subspace. By \cite[Remark~9.2(vii)]{BLPW12a}, $A$ is superregular
at $c$.
\end{proof}

\begin{lemma}\label{l:0406b}
Let $a\in A$. Then
\begin{equation}
\bigcup_{I(a)\subseteq J\in\mcJ_s} A_J=\supp(a)+\menge{x\in X}{\|x\|_0\leq s-\|a\|_0}.
\end{equation}
\end{lemma}
\begin{proof}
``$\subseteq$'': Let $z\in A_J$, where $I(a)\subseteq J\in\mcJ_s$.
Write $J=I(a)\dotcup K$, where $K := J\smallsetminus I(a)$ and
the union is disjoint.
Then $z=y+x$, where $y\in A_{I(a)}=\supp(a)$, $x\in A_{K}$,
and $\|x\|_0 \leq \card(K)=\card(J)-\card(I(a))=s-\|a\|_0$.

``$\supseteq$'':
Let $x\in X$ be such that $\|x\|_0\leq s-\|a\|_0$,
and let $y\in\supp(a)$.
By Lemma~\ref{l:0314a},
$I(y)\subseteq I(a)$,
$I(x+y)\subseteq I(x)\cup I(y)\subseteq I(x)\cup I(a)$ and
$\|x+y\|_0\leq\|x\|_0+\|y\|_0\leq (s-\|a\|_0)+\|a\|_0=s$.
Hence, there exists $J\in\mcJ_s$ such that $I(x)\cup I(a)\subset J$,
and therefore $x+y\in A_{I(x)\cup I(a)}\subseteq A_J$.
\end{proof}


\begin{theorem}[tangent cone to $A$]
Let $a=(a_1,\ldots,a_\dm)\in A$. Then
\begin{equation}\label{e:T_A}
T_A(a)=\bigcup_{I(a)\subseteq J\in\mcJ_s} A_J
=\supp(a)+\menge{x\in X}{\|x\|_0\leq s-\|a\|_0};
\end{equation}
consequently,
\begin{equation}
\|a\|_0=s\quad\Leftrightarrow\quad T_A(a) = A_{I(a)} = \supp(a).
\end{equation}
\end{theorem}
\begin{proof}
Set
\begin{equation}
\rho:=\min_{j\in I(a)}|a_j|>0
\quad\text{and}\quad
A(a):=\bigcup_{I(a)\subseteq J\in\mcJ_s} A_J=\bigcup_{a\in A_J, J\in\mcJ_s} A_J.
\end{equation}
Lemma~\ref{l:0406a}\ref{l:0406a-i} and Lemma~\ref{l:0406c} imply
\begin{equation}
T_A(a)=T_{A\cap\ball{a}{\rho}}(a)=T_{A(a)\cap\ball{a}{\rho}}(a)=T_{A(a)}(a).
\end{equation}
On the other hand, by  Lemma~\ref{l:0406a}\ref{l:0406a-iii},
$T_{A(a)}(a)=A(a)$.
Altogether, $T_A(a) = A(a)$ and we have established the first equality
in \eqref{e:T_A}.
The second equality is precisely Lemma~\ref{l:0406b}.
Finally, the ``consequently'' part is clear from \eqref{e:T_A}.
\end{proof}

\begin{remark}
\label{r:easycones}
For the affine set $B$, the normal and tangent cones are much simpler to
derive: indeed,
because $\pa(B) = \ker M$, it follows that
$T_B(x) = \ker M$ and $N_B(x)=(\ker M)^\perp = \ran M^T$,
for every $x\in B$.
\end{remark}

\begin{remark}[transversality]
Recall \eqref{e:sparse feas} and assume that $c\in A\cap B$.
By \eqref{e:T_A}, Remark~\ref{r:easycones},
and e.g.\ \cite[Lemma~1.43(i)]{BC2011},
we have the implications
\begin{subequations}
\label{e:120412a}
\begin{align}
T_A(c)+T_B(c)=\RR^\dm
&\Leftrightarrow \Bigg(\bigcup_{I(c)\subseteq J\in \mcJ_s}A_J\Bigg) +
\ker(M) = \RR^\dm\\
&\Leftrightarrow \bigcup_{I(c)\subseteq J\in \mcJ_s}\big( A_J +
\ker(M)\big) = \RR^\dm\\
&\Leftrightarrow \inte\Bigg(\bigcup_{I(c)\subseteq J\in \mcJ_s}\big( A_J +
\ker(M)\big)\Bigg) = \RR^\dm\\
&\Rightarrow \overline{\inte\Bigg(\bigcup_{I(c)\subseteq J\in \mcJ_s}\big( A_J +
\ker(M)\big)\Bigg)} =
\overline{\bigcup_{I(c)\subseteq J\in \mcJ_s}\inte\big( A_J +
\ker(M)\big)} = \RR^\dm.
\end{align}
\end{subequations}
Let us assume momentarily that $T_A(c)+T_B(c)=\RR^\dm$.
By \eqref{e:120412a}, there exists $J\in\mcJ_s$ such that
$I(c)\subseteq J$ and $A_J+\ker(M)=\RR^\dm$.
Hence $s+\dim\ker(M) = \dim A_J+\dim \ker(M) \geq \dim(A_J+\ker(M))
=\dim \RR^\dm =\dm=\dim\ker(M)+\rank(M)$.
We have established the implication
\begin{equation}
\label{e:120412b}
T_A(c)+T_B(c)=\RR^\dm
\quad\Rightarrow\quad
s\geq \rank(M);
\end{equation}
that is, \emph{transversality} imposes a lower bound on $s$
and is thus at odds with the objective of finding the \emph{sparsest}
points in $A\cap B$.
\end{remark}

\subsection*{The MAP for the sparse feasibility problem}

We begin with an example illustrating shortcomings of previous approaches.

\begin{example}\label{ex:concrete1}
Suppose that
\begin{equation}
M=\begin{pmatrix}1&1&1\\1&1&0 \end{pmatrix},
\quad p=\begin{pmatrix} 1\\1\end{pmatrix},
\quad\text{and}\quad
s = 1;
\end{equation}
thus,
$m=2$ and $\dm=3$.
Then $B = (1,0,0)+\RR(-1,1,0)$ and hence the set of all solutions to \eqref{e:sparse feas}
consists of\,\footnote{When there is no cause for confusion,
we shall write column vectors as row vectors for space reasons.}
$x^*:=(1,0,0)$ and $y^* := (0,1,0)$.
Since $\|x^*\|=\|y^*\|=s$,
Theorem~\ref{t:Mnor-spr} yields
\begin{equation}
\label{e:N_A(x/y)}
   N_A(x^*) = \{0\}\times\RR\times\RR
\quad\text{and}\quad
N_A(y^*) = \RR\times\{0\}\times\RR.
\end{equation}
On the other hand,
$(\forall x\in B)$ $N_B(x) = \ran M^T
=\lspan\{(1,1,1),(1,1,0)\}$ by Remark~\ref{r:easycones}.
Altogether,
\begin{equation}\label{e:eg1}
 N_A(x^*)\cap\big(-N_B(x^*)\big) = N_A(y^*)\cap\big(-N_B(y^*)\big)
 = \{0\}\times\{0\}\times\RR \neq\{(0,0,0)\}.
\end{equation}
Consequently, neither the Lewis-Luke-Malick framework \cite{LLM} nor
the framework proposed in \cite{Luke12} is able to deal with this case.
Furthermore, in view of \eqref{e:120412b},
the transversality condition
\begin{equation}\label{e:eg1b}
T_A(c)+T_B(c)=\RR^\dm
\end{equation}
proposed by Lewis and Malick \cite{LM}
also always fails because $s=1\not\geq 2=\rank(M)$.

Finally, readers familiar with sparse optimization will also note that the usual
sufficient conditions for the correspondence of solutions to the
nonconvex problem to those of convex relaxations---namely the
restricted isometry property \cite{CandesTao04} or
the mutual coherence condition  \cite{DonohoElad03}---are not satisfied
either.
Constraint qualifications as developed in the present work have no
apparent relation to conditions like restricted isometry or mutual
coherence conditions used to guarantee the correspondence between
solutions to convex surrogate problems and solutions to the problem with
the original $\|\cdot\|_0$ objective.  Indeed,
if the matrix M is changed for instance to 
\begin{equation}
   \begin{pmatrix}1&1&1\\1&2&0
     \end{pmatrix}
\end{equation}
the mutual coherence condition is satisfied and a unique sparsest solution exists,
but still the constraint qualifications \eqref{e:eg1} and \eqref{e:eg1b} are not satisfied.
\end{example}

We are now ready for our main result, which is very general and which
in particular is applicable to the setting of Example~\ref{ex:concrete1}.

\begin{theorem}[main result for sparse affine feasibility and linear local
convergence of MAP]~\\
\label{t:mainsparse}
Let $A$, $\mcA$, $\wt{\mcA}$ $B$, $\mcB$ and $\wt{\mcB}$ be defined by \eqref{e:russell}.
Suppose that $s\leq\dm-1$, that $c\in A\cap B$, and fix $\dd\in\left]0,\overline\dd\right[$~
for $\overline\dd:=\tfrac{1}{3}\min\menge{d_{A_J}(c)}{c\not\in A_J,J\in \mcJ_s}$.
Then
\begin{equation}\label{e:ddbar}
\overline\dd=\tfrac{1}{3}\min\menge{|c_j|}{j\in I(c)} 
\end{equation}
and
\begin{equation}\label{e:theta-spr}
\overline\alpha=\overline\theta=
\theta_{3\dd}(\mcA,\wt{\mcA},\mcB,\wt{\mcB})
=\max\menge{c(A_J,B)}{c\in A_J, J\in\mcJ_s}<1,
\end{equation}
where $\theta_{3\dd}$, $\overline{\theta}$, $\overline{\alpha}$
denote the joint-CQ-number, the limiting joint-CQ-number and 
the exact joint-CQ-number (\eqref{e:jCQn}, \eqref{e:ljCQn} and  
\eqref{e:0217b} respectively) at $c$ associated with 
$(\mcA,\wt{\mcA},\mcB,\wt{\mcB})$.
Suppose the starting point of the MAP $b_{-1}$ satisfies
$\|b_{-1}-c\|\leq\frac{(1-\overline{\theta})\dd}{6(2-\overline{\theta})}$.
Then $(a_k)_\kkk$ and $(b_k)_\kkk$ converge linearly
to some point in $\bar{c}\in A\cap B\cap\ball{c}{\dd}$
with rate $\overline{\theta}^2$.
\end{theorem}
\begin{proof}
Observe that \eqref{e:ddbar} follows from Lemma~\ref{l:120412c}.
Let $J\in\mcJ_s$.
If $c\not\in A_J$,
then $\ball{c}{3\dd}\cap A_J=\emp$ and hence
$\theta_{3\dd}(A_J,A_J,B,B)=\minf$.
On the other hand, if $c\in A_J$,
then $c\in A_J\cap B$ and hence
$\theta_{3\dd}(A_J,A_J,B,B)=c(A_J,B)<1$
by Theorem \ref{t:CQn=c}.  
Combining this with Theorem~\ref{t:CQ1}\ref{t:CQ1iv}, we obtain \eqref{e:theta-spr}.
Because $A_J$ is a linear subspace and hence convex,
Proposition~\ref{p:jsreg} yields the $(0,\pinf)$-joint-regularity of
$\mcA$; in particular, $\mcA$ is $(\wt{B},0,3\dd)$-joint-regular.
Analogously, $\wt{B}=(B)$ is $(\wt{A},0,3\dd)$-joint-regular.
Now apply Theorem~\ref{p:jdb} to complete the proof.
\end{proof}

\begin{remark}
Some comments regarding Theorem~\ref{t:mainsparse} are in order.
\begin{enumerate}
\item Note that regularity of the intersection is not an assumption of the theorem, but is 
rather {\em automatically} satisfied.  This is in contrast to the results of \cite{LM} and 
\cite{LLM} where the required regularity is assumed to hold.   
In view of Example~\ref{ex:concrete1}, which illustrated that the
notions of regularity developed in \cite{LM} and \cite{LLM} are {\em not} 
satisfied, it is clear
that Theorem~\ref{t:mainsparse} is new and has a genuinely wider
range of applicability.
  
\item
In contrast to \cite{LLM} and \cite{LM}, our analysis yields a
quantification of the neighborhood on which local linear
convergence is guaranteed.
\item Finding the local neighborhood on which linear convergence is
guaranteed is not an easy task, and may well be tantamount of finding
the sparsest solution; however, it does open the door to justify
combining the MAP with more aggressive algorithms such as Douglas-Rachford
in order to find such neighborhoods.
\item Consider again Example~\ref{ex:concrete1} and its notation.
Since $s=1$, $\mcA = (A_1,A_2,A_3)$, where $A_i = \RR e_i$,
while $B = e_1 + \RR(e_2-e_1)$.
Hence $c(A_1,B)=c(\RR e_1,\RR(e_2-e_1)) =
|\tscal{e_1}{(e_2-e_1)/\sqrt{2}}|=1/\sqrt{2}$
by Theorem~\ref{t:CQn=c} and Corollary~\ref{p:CQn2l}.
Similarly, $c(A_2,B)=1/\sqrt{2}$ while $A_3\cap
B=\varnothing$.
Let $c\in\{x^*,y^*\}$.
Then $\overline{\theta}=1/\sqrt{2}$ and
\eqref{e:ddbar} implies that
$\overline{\dd}=1/3$.
The predicted rate of linear convergence is
$\overline{\theta}^2=1/{2}$.
\item The projectors $P_A$ and $P_B$ given by 
\eqref{e:projAU} and
\eqref{e:projAx=b} are easy to implement numerically, 
which we have done.  Indeed, for random initial guesses $b_{-1}$ in the neighborhood
$\ball{c}{(\sqrt{2}-1)/(18(2\sqrt{2}-1))}$
the observed ratios $\|a_{k+1}-c\|/\|a_{k}-c\|$ and $\|b_{k+1}-c\|/\|b_{k}-c\|$ for 
$a_k=P_A b_k$ 
($k\in\NN$, $b_0=P_B b_{-1}$) and 
$b_k = P_B a_{k-1}\in B$ ($k\in\NN\setminus\{0\}$) are $1/2 + |O(10^{-13})|)$.  
The observed rate 
corresponds nicely to the theory under the 
assumption of exact evaluation of the projections.  However, exact projections are not in fact computed 
in practice (in particular the projection onto the affine set $B$), so the numerical illustration is not precisely applicable. 
Inexact alternating projections is beyond the scope of this work.    
\end{enumerate}
\end{remark}


\section*{Conclusion}

We have applied new tools in variational analysis to the problem of finding sparse vectors
in an affine subspace.   
The key tool is the restricted normal cone which generalizes classical
normal cones. 
The restricted normal cones are used to define constraint qualifications, and 
notions of regularity that provide sufficient conditions for local convergence of iterates of the 
elementary method of alternating projections applied to the lower level sets of the 
function $\|\cdot\|_0$ and an affine set.  
Key ingredients were suitable restricting sets
$(\wt{\mcA}$ and $\wt{\mcB})$.
The coarsest choice,
$(\wt{A},\wt{B})=(X,X)$,
recovers the framework by Lewis, Luke, and Malick \cite{LLM}.
We show, however, that the corresponding regularity conditions are not 
satisfied in general for the sparse feasibility problem \eqref{e:sparse feas}.  
The tighter (and hence more powerful) 
choice of $(\wt{\mcA},\wt{\mcB})=(\mcA,\mcB)$ recovers local 
linear convergence and yields an estimate of the radius of convergence.

\subsection*{Acknowledgments}
HHB was partially supported by the Natural Sciences and Engineering
Research Council of Canada and by the Canada Research Chair Program.
This research was initiated when HHB visited the
Institut f\"ur Numerische und Angewandte Mathematik,
Universit\"at G\"ottingen because of his study leave
in Summer~2011. HHB thanks DRL and the Institut for their hospitality.
DRL was supported in part by the German Research Foundation grant SFB755-A4. 
HMP was partially
supported by the Pacific Institute for the Mathematical Sciences and
and by a University of British Columbia research grant. XW was
partially supported by the Natural Sciences and Engineering Research
Council of Canada.


\end{document}